\documentclass[a4paper,  leqno,  myheadings,  11pt,  twoside]{article}
\usepackage{mathrsfs}

\usepackage{amsmath,  amsthm,  amssymb,  amscd,  amsxtra, graphicx}
\usepackage{latexsym,  amsfonts, enumerate}

\DeclareMathOperator*{\esssup}{esssup}

\def\qc{quasi\-conformal }

\def\T{Teich\-m\"ul\-ler }

\def\bmu{\emu_Z}

\def\nmu{\|\mu\|_\infty}

\def\nmu{\|\mu\|_\infty}

\def\wt{\widetilde}
\def\vp{\varphi}

\def\vp{\varphi}

\def\pa{\partial}

\def\ov{\overline}

\def\de{\Delta}
\def\q1s{Q^1(S)}

\def\ts{T(S)}

\def\zs{Z(S)}
\def\azs{AZ(S)}
\def\azde{AZ(\de)}

\def\ats{AT(S)}
\def\qde{Q^1(\de)}
\def\qds{Q^1_d(S)}
\def\qdde{Q^1_d(\de)}
\def\tde{T(\de)}

\def\atde{AT(\de)}
\def\dat{d_{AT}}
\def\oo{[[0]]}
\def\oaz{\oo_{AZ}}

\def\emu{[\mu]}

\def\emb{[\mu]_Z}

\def\azmu{[[\mu]]_{AZ}}

\def\mut{[[\mu]]}

 \makeatletter
\@addtoreset{equation}{section}

\newtheorem{theorem}{Theorem}
\newtheorem{cor}{Corollary}[section]
\newtheorem{lemma}{Lemma}[section]

\newtheorem*{rem}{Remark}

\begin{document}

\title{\bf{On nonuniqueness  of
geodesics in asymptotic \T space
 }}
\author{GUOWU YAO
}
\date{}
\maketitle
\begin{abstract}\noindent

 In an infinite-dimensional \T space, it is known that
  the  geodesic connecting two points can be unique or not.
   In this paper, we study the situation on the geodesic  in the universal
    asymptotic \T space $\atde$.  We introduce the notion of substantial point and non-substantial point in $\atde$. The set of all non-substantial points is open and dense in $\atde$. It is shown that
there are infinitely many geodesics joining a non-substantial point to the
basepoint.  Although we have difficulty in dealing with  the substantial points,
 we give an example to show that there are infinitely many geodesics connecting certain
     substantial point  and the basepoint. It is also shown that there  are always infinitely many straight lines containing two points in $\atde$.
     Moreover, with the help of the Finsler structure on the asymptotic \T space, a variation formula for the asymptotic \T  metric is obtained.

\end{abstract}
\renewcommand{\thefootnote}{}

\renewcommand{\thefootnote}{}
\footnote{The   work  was
supported by the National Natural Science Foundation of China
(Grant No. 11271216).}

\footnote{{2010 \it{Mathematics Subject Classification.}} Primary
30C75,  30C62.} \footnote{{\it{Key words and phrases.}} \T space, asymptotic \T space,
  geodesic, substantial boundary point.}

\begin{centering}\section{\!\!\!\!\!{. }
Introduction}\label{S:intr}\end{centering}

Let $S$ be a hyperbolic Riemann surface,  that is,  it is  covered
by a holomorphic map: $\varpi: \de\to S$,  where $\de=\{|z|<1\}$ is the
open unit disk. Let $T(S)$ be the \T space of $S$.   A quotient space of the  \T space $T(S)$, called the asymptotic \T
space and denoted by $AT(S)$, was introduced by Gardiner and
Sullivan (see~\cite{GS} for $S=\de$ and by Earle, Gardiner and
Lakic for arbitrary hyperbolic
$S$~\cite{EGL0,EGL,GL}).

 $\ats$ is interesting only when $\ts$ is infinite dimensional,
which occurs when $S$ has border or when $S$ has infinite
topological type, otherwise, $\ats$ consists of just one point. In
recent years, the asymptotic spaces are extensively studied, for examples, one can refer to
\cite{EGL0,EGL,EMS, Fuj, Mats, Mi, Yao}.

We shall use some geometric terminologies  adapted from \cite{Bus}
by Busemann. Let $X$ and $Y$ be metric spaces. An isometry of $X$
into $Y$ is a distance preserving map. A straight line in $Y$ is a
(necessarily closed) subset $L$ that is  an isometric image of the
real line $\mathbb{R}$.   A geodesic in $Y$ is an isometric image of
a non-trivial compact interval of $\mathbb{R}$. Its endpoints are
the images of the endpoints of the interval, and we say that the
geodesic joins its endpoints.

 Geodesics play an important role in the theory
 of \T spaces.
In an finite-dimensional \T space $\ts$, there is always a unique
geodesic connecting  two points. The situation  is
substantially different in an infinite dimensional \T space (see
\cite{ELi,HuSh,Li, Li2, Li3,Ta}). Generally,  for a Strebel point,
   there is a unique geodesic connecting it and the basepoint. A natural
question is whether  the geodesic connecting  two points in $\ats$
is unique. In \cite{Fan}, by a lengthy computation Fan  gave certain examples to
show the nonuniqueness of geodesics in $\ats$.

The motivation of this paper  is to investigate the nonuniqueness of geodesic in the universal asymptotic \T space $\atde$. We introduce the notion of   substantial point and non-substantial point in $\atde$. The set of all non-substantial points is open and dense in $\atde$. The first main result is the following theorem.

\begin{theorem} \label{Th:geodesic}For every non-substantial point in $\atde$, there are
 infinitely many geodesics
joining it with the basepoint.
\end{theorem}
   We  have some difficulty  in dealing with  the
   substantial points.  Nevertheless, we
   give an example to show that there are infinitely many geodesics connecting certain
     substantial point  and the basepoint which might support the conjecture that there
     are always infinitely many geodesics connecting  two points
     in $\atde$.

    In a  finite-dimensional \T space, there is a unique straight
line passing through two   points \cite{Kra}.
In an infinite-dimensional \T space, the work in \cite{ELi} shows that if and only if $\tau$ is a Strebel point,
there is a unique straight line passing through $\tau$ and the basepoint in $\ts$.
    The second main result of this paper  characterizes the nonuniqueness of straight lines containing two points in $\atde$.

    \begin{theorem}\label{Th:line} For any two points in $\atde$,  there
     are  infinitely many straight lines containing them.\end{theorem}

This paper is organized as follows. In Section \ref{S:prel}, we introduce some basic notion in  the \T space theory. A variation formula for the asymptotic \T metric is obtained  in Section \ref{S:variation}.  Theorems \ref{Th:geodesic} and \ref{Th:line} are proved in Section \ref{S:geodesic} and Section \ref{S:line} separately. We investigate the relationship on
the substantial boundary points for points along a geodesic in Section \ref{S:property}. In Section \ref{S:example},  an example is given to show the nonuniqueness of geodesics joining certain substantial point with the basepoint.
Some parallel results  in the infinitesimal setting are obtained in the last section.

The method used here can also be used to deal with some more general cases. However, there are some difficulties in solving the problem in all cases.

\begin{centering}\section{\!\!\!\!\!{. }Some Preliminaries}\label{S:prel}\end{centering}

\subsection{\T space and asymptotic \T space}
Let $S$ be a Riemann surface of topological type.  The \T space
$\ts$ is the space of equivalence classes of \qc maps $f$ from $S$
to a variable Riemann surface $f(S)$. Two \qc maps $f$ from $S$ to
$f(S)$ and $g$ from $S$ to $g(S)$ are equivalent if there is a
conformal map $c$ from $f(S)$ onto $g(S)$ and a homotopy through \qc
maps $h_t$ mapping $S$ onto $g(S)$ such that $h_0=c\circ f$, $h_1=g$
and $h_t(p)=c\circ f(p)=g(p)$ for every $t\in [0,1]$ and every $p$
in the ideal boundary of $S$. Denote by  $[f]$  the \T equivalence
class of $f$; also sometimes denote the equivalence class by $\emu$
where $\mu$ is the Beltrami differential of $f$.

The asymptotic \T space is the space of a larger equivalence classes. The definition of the new equivalence classes
 is exactly
the same as the previous definition with one exception; the word
\emph{conformal } is replaced by \emph{asymptotically conformal}. A
\qc map $f$ is asymptotically conformal if for every $\epsilon>0$,
there is a compact subset $E$ of $S$, such that the dilatation of
$f$ outside of $E$ is less than $1+\epsilon$. Accordingly, denote by
$[[f]]$ or $\mut$ the asymptotic equivalence class of $f$.

Denote by $Bel(S)$ the Banach space of Beltrami differentials
$\mu=\mu(z){d\bar z}/dz$ on $S$ with finite $L^{\infty}$-norm and by
$M(S)$ the open unit ball in $Bel(S)$.

For $\mu\in M(S)$, define
\begin{equation*}
k_0(\emu)=\inf\{\|\nu\|_\infty:\,\nu\in\emu\}.
\end{equation*}

Define $h^*(\mu)$ to be the infimum over all
 compact subsets $E$ contained in $S$ of the essential supremum
 norm of the Beltrami differential $\mu(z)$ as $z$ varies over
 $S\backslash E$ and  $h(\emu)$ to be the infimum of
 $h^*(\nu)$ taken over all representatives $\nu$ of the class
 $\emu$.  It is obvious that $h(\emu)\leq
 k_0(\emu)$. Following \cite{ELi}, $\emu$ is called
 a Strebel point if $h(\emu)<k_0(\tau)$; otherwise, $\tau$ is called
 a non-Strebel point.

Put
\begin{equation*}
h(\mut)=\inf\{h^*(\nu):\,\nu\in\mut\}.
\end{equation*}
 We say that $\mu$ is extremal  in $\emu$
  if $\nmu=k_0(\emu)$ and $\mu$ is asymptotically extremal if $h^*(\mu)=h(\mut)$. The relation $h(\emu)=h(\mut)$
   is due to the definition.

The \T metric $d_{T}$ between two points $\tau, \sigma \in \ts$ is
defined as follows:
\begin{equation*}
d_{T}(\tau, \sigma)=\frac{1}{2}\inf_{\mu\in\tau,\ \nu\in
\sigma}\log\frac{1+\|(\mu-\nu)/(1-\bar\nu\mu)\|_\infty}{1-\|(\mu-\nu)/(1-\bar\nu\mu)\|_\infty}.
\end{equation*}
The asymptotic \T metric $d_{AT}$ between two points $\wt\tau,
\wt\sigma \in \ats$ is defined by
\begin{equation*}
d_{AT}(\wt\tau,\wt \sigma)=\frac{1}{2}\inf_{\mu\in\wt\tau,\ \nu\in
\wt\sigma}\log\frac{1+\|(\mu-\nu)/(1-\bar\nu\mu)\|_\infty}{1-\|(\mu-\nu)/(1-\bar\nu\mu)\|_\infty}.
\end{equation*}
In particular, the  distance between $\mut$ and the basepoint $\oo$ is
\begin{equation*}
d_{AT}(\mut, \oo)=\frac{1}{2}\log H(\mut),\;\text {where\; }H(\mut)=\frac{1+h(\mut)}{1-h(\mut)}.
\end{equation*}

\subsection{Tangent spaces to \T space and asymptotic \T space}

The cotangent space to $\ts$ at the basepoint is the Banach space
$Q(S)$ of integrable holomorphic quadratic differentials $\vp$ on $S$ with
$L^1-$norm
\begin{equation*}
\|\vp\|=\iint_{S}|\vp(z)|\, dxdy<\infty.
 \end{equation*}    In what follows,  let $\q1s$
denote the unit sphere of $Q(S)$. Moreover, let $\qds$ denote the
set of all degenerating sequence $\{\vp_n\}\subset \q1s$. By
definition, a sequence $\{\vp_n\}$ is called degenerating if it
converges to 0 uniformly on compact subset of $S$.

Two Beltrami differentials $\mu$ and $\nu$ in $Bel(S)$ are said to
be infinitesimally equivalent if
\begin{equation*}\iint_S(\mu-\nu)\vp \, dxdy=0,  \text{
for any } \vp\in Q(S).
\end{equation*}
The tangent space $\zs$ of $\ts$ at the basepoint is defined as the
set of the quotient space of $Bel(S)$ under the equivalence
relations.  Denote by $\emb$ the equivalence class of $\mu$ in
$\zs$.

$\zs$ is a Banach space and actually \cite{GL}  its standard
sup-norm satisfies
\begin{equation*}
\|\emb\|:=\sup_{\vp\in \q1s} \left|\iint_S \mu\vp \,
dxdy \right|=\inf\{\|\nu\|_\infty:\,\nu\in\emb\}.\end{equation*}

Two Beltrami differentials $\mu$ and $\nu$ in $Bel(S)$ are said to
be infinitesimally asymptotically equivalent if
\begin{equation*}\sup_{\qds}\limsup_{n\to\infty}\left|\iint_S(\mu-\nu)\vp_n \, dxdy\right|=0,
\end{equation*}
where the first  $supremum$ is taken when  $\{\vp_n\}$ varies over
$\qds$.

 The tangent space $\azs$ of $\ats$ at the basepoint is
defined as the set of the quotient space of $Bel(S)$ under the asymptotic
equivalence relation.  Denote by $\azmu$ the equivalence class of
$\mu$ in $\azs$.

Define  $b(\bmu)$ to be  the infimum over all elements in the
equivalence class $\bmu$  of the quantity $b^*(\nu)$. Here
$b^*(\nu)$ is the infimum over all compact subsets $E$ contained in
$S$ of the essential supremum of the the Beltrami differential $\nu$
as $z$ varies over $S-E$. It is obvious that $b^*(\mu)\leq\|\bmu\|
$. $\bmu$ is called
 an infinitesimal Strebel point if $b(\bmu) <\|\emb\|$.

Put
\begin{equation*}
b(\azmu)=\inf\{b^*(\nu):\,\nu\in\azmu\}.
\end{equation*}

 We say that $\mu$ is (infinitesimally) extremal
  if $\nmu=\|\emb\|$ and $\mu$ is (infinitesimally) asymptotically extremal if $b^*(\mu)=b(\azmu)$.
 We also have $b(\emb)=b(\azmu)$ \cite{GL}.

$\azs$ is a Banach space and   its  standard infinitesimal
asymptotic norm satisfies (see \cite{GL})
\begin{equation*}
\|\azmu\|:=\sup_{\qds}\limsup_{n\to\infty} \left|\iint_S \mu\vp_n \,
dxdy\right|=\inf\{\|\nu\|_\infty:\,\nu\in\azmu\}=b(\azmu).\end{equation*}

\subsection{Substantial boundary points and Hamilton sequence}

Now we define the notion of boundary dilatation of a \qc mapping at
a boundary point. For a Riemann surface, the meaning of what is a
boundary point can be problematic. However, if  $S$ can be embedded into a larger surface $\wt S$ such that the closure of
$S$ in $\wt S$ is compact, then it is possible to define the
boundary dilatation. In this section, we assume that $S$ is such a surface.

Let $p$ be a point on $\pa S$ and let $\mu\in Bel(S)$. Define
 \begin{equation*}
 h^*_p(\mu)=\inf\{\esssup_{z\in U\bigcap S}|\mu(z)|:\;U \text{ is an open
 neighborhood
 in } \wt{S} \text{ containing }p\}
 \end{equation*}
 to be the boundary dilatations of $\mu$ at  $p$.
If $\mu\in M(S)$, define
 \begin{equation*}
 h_p([\mu])=\inf\{h^*_p(\nu):\;\nu\in[\mu]\}
 \end{equation*}
to be the boundary dilatations $[\mu]$ at  $p$. For a general $\mu\in
Bel(S)$, define
 \begin{equation*}
 b_p([\mu]_Z)=\inf\{h^*_p(\nu):\;\nu\in \emb\}
 \end{equation*}
to be the boundary dilatations of $\emb$ at  $p$. If we define the
quantities
 \begin{equation*}
 h_p(\mut)=\inf\{h^*_p(\nu):\;\nu\in \mut\},\quad b_p(\azmu)=\inf\{h^*_p(\nu):\;\nu\in
 \azmu\},
 \end{equation*}
then  $h_p(\emu)=h_p(\mut)$ and $b_p(\emb)=b_p(\azmu)$. In particular,  Lakic \cite{La} proved that when $S$ is a plane domain,
\begin{equation*}
h(\mut)=\max_{p\in \pa S} h_p(\mut),\quad b(\azmu)=\max_{p\in \pa S}
 b_p(\azmu).
 \end{equation*}

 As is well
known, $\mu$ is extremal if and only if
  it  has a  so-called Hamilton sequence, namely, a sequence
$\{\psi_n\}\subset \q1s$, such that
\begin{equation}
\lim_{n\to\infty}\iint_S \mu\psi_n(z)dxdy=\nmu.
\end{equation}
Similarly, by Theorem 8 on page 281 in \cite{GL}, $\mu$ is asymptotically extremal if and only if
  it  has an asymptotic Hamilton sequence, namely, a degenerating  sequence
$\{\psi_n\}\subset \q1s$, such that
\begin{equation}
\lim_{n\to\infty}\iint_S \mu\psi_n(z)dxdy=h^*(\mu).
\end{equation}

Now, we assume that $S$ is a plane domain with two or more boundary points. Then, the following lemma derives from  Theorem 6 on page 333 in ~\cite{GL}:
 \begin{lemma}\label{Th:atslem1}
The following three conditions are equivalent for every boundary point
$p$ of $S$ and every asymptotic  or infinitesimal asymptotic extremal representative $\mu$:
\\(1) $h(\emu)=h_p(\emu)$ (equivalently, $h(\mut)=h_p(\mut)$),\\
(2) $b(\emu)=b_p(\emu)$   (equivalently,  $b(\azmu)=b_p(\azmu)$),\\
(3) there exists an asymptotic Hamilton sequence for $\mu $ degenerating towards
$p$, i.e., a sequence
$\{\psi_n\}\subset\q1s$   converging  uniformly to 0 on compact subsets of $S\backslash\{p\}$, such that
\begin{equation}
\lim_{n\to\infty}\iint_S \mu\psi_n(z)dxdy=h_p^*(\mu).
\end{equation}
 \end{lemma}
 If one  of three conditions in the  lemma holds at some $p\in
\pa S$, we call $p$ is a substantial boundary point for $\mut$ (or
$\emu$) and $\azmu$ (or $\emb$), respectively.

\begin{centering}\section{\!\!\!\!\!{. }
A variation formula for asymptotic \T metric}\label{S:variation}\end{centering}

In this section, suppose that  $S$ is  a hyperbolic Riemann surface such that $\ats$ is not trivial.
Let $\Phi=\Phi_S:\; M(S)\to \ats$ denote the canonical projection from $M(S)$ to $\ats$ so that $\Phi(\mu)$ is the asymptotic equivalence class $\mut$ for $\mu\in M(S)$.
As is well known, the \T distance on $\ts$ is induced by a Finsler structure (see \cite{Byr, Gar,  Roy}). Proposition 5 in \cite{EGL} shows that  a quotient Finsler structure on $\ats$ is induced by the \T Finsler structure on $M(S)$ and the map $\Phi$,  and hence
the asymptotic \T distance on $\ats$ is induced by the quotient Finsler structure. For $\mu\in M(S)$ and $\nu\in Bel(S)$, the quotient Finsler structure $F=F_S$ is
\begin{align}
F(\Phi(\mu),\Phi'(\mu)\nu)=\inf\left\{\left\|\frac{\sigma}{1-\|\mu\|^2}\right\|_\infty:\; \sigma\in Bel(S) \text{ and } \Phi'(\mu)\sigma=\Phi'(\mu)\nu\right\}.
\end{align}
In particular, $F$ is positive definite and locally Lipschitz continuous.

For any $\mu\in M(S)$, set

\begin{align*}
& I(\mu)=I_S(\mu)=\sup_{\qds}\limsup_{n\to\infty}\left |Re\iint_S \frac{\mu\vp}{1-|\mu|^2} \,
dxdy \right|,\\
& J(\mu)=J_S(\mu)=\sup_{\qds}\limsup_{n\to\infty} \left|
\iint_S \mu\vp \,
dxdy \right|,\\
& \de(\mu)=J_S(\mu)=\sup_{\qds}\limsup_{n\to\infty} \left|\iint_S \frac{|\mu|^2|\vp|}{1-|\mu|^2} \,
dxdy \right|.
\end{align*}
 The Reich-Strebel inequalities (see \cite{RS3} or Theorem 9 of Chapter 4 in \cite{GL}), so-called fundamental inequalities,  play a very important role in study of the theory of \T spaces. The asymptotic analogues of the Reich-Strebel inequalities are obtained in \cite{EGL}.

\noindent\textbf{The Asymptotic  Fundamental Inequalities.}
Suppose $f$ is a \qc mapping
from $S$ to $S^\mu=f(S)$ with $\mu$ its Beltrami differential. Let $H=H(\mut)$.
Then

\begin{flushleft}\begin{equation*}\label{Eq:asymineqI}
(I) \quad \frac{1}{H}\leq\inf_{\qds}
\liminf_{n\to\infty}\iint_S\frac{\left|1-\mu\frac{\phi_n}{|\phi_n|}\right|^2}{1-|\mu|^2}|\phi_n |\, dxdy,
\end{equation*}
\end{flushleft}
 and
\begin{flushleft}\begin{equation*}\label{Eq:asymineqII}
(II) \quad H\leq
\sup_{\qds}\left\{\liminf_{n\to\infty}\iint_S\frac{\left|1-\mu\frac{\phi_n}{|\phi_n|}\right|^2}{1-|\mu|^2}|\phi_n |\ dxdy\right\}.
\end{equation*}
\end{flushleft}
A simple argument shows that the inequality (I) implies

\begin{equation}\label{Eq:ijh1}
I(\mu)\leq \frac{h(\mut)}{1+h(\mut)}+\de(\mu),
\end{equation}
and  the inequality (II) does

\begin{equation}\label{Eq:ijh2}
\frac{h(\mut)}{1-h(\mut)}-\de(\mu)\leq I(\mu).
\end{equation}
Combining (\ref{Eq:ijh1}) and (\ref{Eq:ijh2}), it can be deduced (see Theorem 11 of Chapter 14 in \cite{GL}) that
\begin{equation}
\begin{split}
&F(\Phi(\mu),\Phi'(\mu)\nu)=J\left(\left(\frac{\nu}{1-|\mu|^2}\frac{\pa f}{\ov{\pa f}}\right)\circ f^{-1}\right)\\
&=\sup_{Q^1_d(S^\mu)}\limsup_{n\to\infty}\left|\iint_{S^\mu}\left[\left(\frac{\nu}{1-|\mu|^2}\frac{\pa f}{\ov{\pa f}}\right)\circ f^{-1}\right]\psi_n\,dudv\right|.
\end{split}
\end{equation}

Let $\mu(t)$ be a continuous curve from $[0,t_0]$ into $M(S)$. We say $\mu(t)$ is differentiable at 0 if there exists some $\mu\in Bel(S)$ such
that $\mu(t)=\mu(0)+t\mu+o(t)$ as $t\to 0_+$, or precisely,
\begin{equation*}
\lim_{t\to 0_+}\left\| \frac{\mu(t)-\mu(0)}{t}-\mu\right\|_\infty=0.
\end{equation*}
$\mu$ is called the derivative of $\mu(t)$ at 0 and is denoted by $\mu'(0)$.

The following theorem provides a variation formula for the asymptotic \T metric.

\begin{theorem}\label{Th:binaryasym}
Let $\mu(t)$ and $\nu(t)$ be two continuous curves from $[0,t_0]$ into $M(S)$ which are differentiable at 0 and satisfy $\mu(0)=\nu(0)$. Then,
\begin{equation}\label{Eq:variation}
\dat(\Phi(\mu(t)),\Phi(\nu(t)))=tF(\Phi(\mu(0)), \Phi'(\mu(0))(\mu'(0)-\nu'(0)))+o(t),\; t\to 0_+.
\end{equation}
\end{theorem}
To prove the theorem, we need a lemma.
\begin{lemma}\label{Th:ijhde}
Suppose $S_t$ is a Riemann surface which may depend on $t\in [0,t_0]$. If $\zeta(t)\in M(S_t)$ satisfies $\zeta(t)=t\eta(t)+o(t)$ as $t\to 0_+$, where $\eta(t)\in Bel(S_t)$ satisfies $\eta(t)=O(1)$ as $t\to 0_+$, then it holds that
\begin{equation}\label{Eq:var1}
\dat(\Phi_{S_t}(0),\Phi_{S_t}(\zeta(t)))=tJ_{S_t}(\eta(t))+o(t),\; t\to 0_+.
\end{equation}
\end{lemma}
\begin{proof}
By the definition of the asymptotic \T  distance, we have
\begin{equation*}
\dat(\Phi_{S_t}(0),\Phi_{S_t}(\zeta(t)))=\frac{1}{2}\log\frac{1+h([[\zeta(t)]])}{1-h([[\zeta(t)]])}=(1+o(1))h([[\zeta(t)]]),\;  t\to 0_+.
\end{equation*}
Replace $\mu$ by $\zeta(t)$ and apply the inequalities (\ref{Eq:ijh1}) and (\ref{Eq:ijh2}) on the Riemann surface $S_t$. It is clear that, as $ t\to 0_+$,
\begin{equation*}
I_{S_t}(\zeta(t))=J_{S_t}(\zeta(t)+O(t^2).
\end{equation*}
Notice that both $\frac{h([[\zeta(t)]])}{1-h([[\zeta(t)]])}$ and $\frac{h([[\zeta(t)]])}{1+h([[\zeta(t)]])}$ differ from $h([[\zeta(t)]])$ by a term of order $t^2$ and $\de_{S_t}(\zeta(t))$ is a term of order $t^2$. Therefore, we have
\begin{equation*}
h([[\zeta(t)]])=J_{S_t}(\zeta(t))+o(t), \;\text{ as } t\to 0_+,
\end{equation*}
and (\ref{Eq:var1}) follows.
\end{proof}

\textbf{Proof of Theorem \ref{Th:binaryasym}.} Let $f^{\mu(t)}:\; S\to S^{\mu(t)}$ and  $f^{\nu(t)}:\; S\to S^{\nu(t)}$ be the \qc mappings with the Beltrami differentials $\mu(t)$ and $\nu(t)$ respectively. Denote by $\zeta(t)$ the Beltrami differential of $f^{\mu(t)}\circ (f^{\nu(t)})^{-1}$, that is,
\begin{align*}
\zeta(t)=\left(\frac{\mu(t)-\nu(t)}{1-\ov{\nu(t)}\mu(t)}\frac{\pa f^{\nu(t)}}{\ov{\pa f^{\nu(t)}}}\right)\circ (f^{\nu(t)})^{-1}.
\end{align*}
Since $\mu(t)$ and $\nu(t)$ is differentiable at 0, we have
\begin{align*}
\frac{\mu(t)-\nu(t)}{1-\ov{\nu(t)}\mu(t)}=\frac{t(\mu'(0)-\nu'(0))}{1-|\nu(t)|^2}+o(t),\; t\to 0_+.
\end{align*}
Put $S_t=S^{\nu(t)}$. We find that $\zeta(t)\in M(S_t)$ satisfies the assumption of Lemma \ref{Th:ijhde} with
\begin{align*}
\eta(t)=\left(\frac{\mu'(0)-\nu'(0)}{1-|\nu(t)|^2}\frac{\pa f^{\nu(t)}}{\ov{\pa f^{\nu(t)}}}\right)\circ (f^{\nu(t)})^{-1}.
\end{align*}
By Lemma \ref{Th:ijhde},
\begin{equation*}
\dat(\Phi_{S_t}(0),\Phi_{S_t}(\zeta(t)))=tJ_{S_t}(\eta(t))+o(t),\; t\to 0_+.
\end{equation*}
Due to the continuity of the Finsler structure $F$, as $t\to 0_+$ it derives that
\begin{align*}
 J_{S_t}(\eta(t))&=F(\Phi(\nu(t)), \Phi'(\nu(t))(\mu'(0)-\nu'(0)))\\
 &\to  F(\Phi(\nu(0)), \Phi'(\nu(0))(\mu'(0)-\nu'(0)))\\
 &=  F(\Phi(\mu(0)), \Phi'(\mu(0))(\mu'(0)-\nu'(0))).
\end{align*}
Thus,
\begin{equation*}
\dat(\Phi_{S_t}(0),\Phi_{S_t}(\zeta(t)))=tF(\Phi(\mu(0)), \Phi'(\mu(0))(\mu'(0)-\nu'(0)))+o(t),\; t\to 0_+.
\end{equation*}
Note that the map
 \begin{align*}\mathscr{A}:\ats &\to AT(S_t),\\
 [[f]] &\to [[f\circ (f^{\nu(t)})^{-1}]]
 \end{align*}
is   distance-preserving. We have
$\dat(\Phi(\mu(t)),\Phi(\nu(t)))=\dat(\Phi_{S_t}(0),\Phi_{S_t}(\zeta(t)))$.
Now (\ref{Eq:variation}) follows immediately. The concludes the proof of Theorem \ref{Th:binaryasym}.

The following corollary is an immediate consequence of Theorem \ref{Th:binaryasym}.
\begin{cor}\label{Th:corvar}
For any two Beltrami differentials $\mu$ and $\nu$ in $Bel(S)$, it holds that
\begin{equation}\label{Eq:binaryasym}
\dat([[t\mu]],[[t\nu]])=t\sup_{\qds}\limsup_{n\to\infty} \left|\iint_S(\mu-\nu)\phi_n \, dxdy\right|+o(t),\; t\to 0_+,
\end{equation}
or equivalently,
\begin{equation}\label{Eq:binary1}\lim_{t\to 0^+}\frac{d_{AT}([[t\mu]], [[t\nu]])}{t}
=\sup_{\qds}\limsup_{n\to\infty} \left|\iint_S(\mu-\nu)\phi_n \, dxdy\right|.
\end{equation}
\end{cor}
Generally, the infinitesimal form of the asymptotic \T metric on $\ats$ in the fiber of the tangent space over the identity is given by the infinitesimal asymptotic norm (see Theorem 10 of Chapter 14 in \cite{EGL}). Precisely,
\begin{equation}
\dat([[0]],[[t\mu]])=t\sup_{\qds}\limsup_{n\to\infty} \left|\iint_S\mu\phi_n \, dxdy\right|+o(t),\; t\to 0_+.
\end{equation}
Now, we see that (\ref{Eq:binaryasym}) of Corollary \ref{Th:corvar} gives a binary infinitesimal form of the asymptotic \T metric on $\ats$.

\begin{centering}\section{\!\!\!\!\!{. }Geodesics joining  non-substantial points with the basepoint}\label{S:geodesic}\end{centering}

$\mut$ (or
$\azmu$) is called a substantial point in $\atde $  (or $AZ(\de)$), if every
 $p\in \pa \de$ is a substantial boundary point for $\mut$ (or
$\azmu$); otherwise, $\mut$ (or $\azmu$) is called a non-substantial point.

Let $SP$ and $ISP$
 denote the collection of all (infinitesimal)
substantial points in $\atde$ and $AZ(\de)$, respectively. Since every substantial point can be approximated by a sequence of non-substantial points, it is clear that $\atde\backslash SP$
and $AZ(\de)\backslash ISP$ are open and dense in
 $\atde$ and $AZ(\de)$,
respectively.

Let $d_H(z_1,z_2)$ denote the hyperbolic distance between two points
$z_1$, $z_2$ in  $\de$, i.e.,
\begin{equation*}
d_H(z_1,z_2)=\frac{1}{2}\log\frac{1+|\frac{z_1-z_2}{1-\bar
z_1z_2}|} {1-|\frac{z_1-z_2}{1-\bar z_1z_2}|}.
\end{equation*}

To prove Theorem \ref{Th:geodesic}, we need a series of lemmas.

\begin{lemma}\label{Th:dist}
Let $t_1$, $t_2$ be two complex numbers and $k_1$, $k_2$ be
two real numbers. Then we have
\begin{equation}\label{Eq:dist}
\left|\frac{(t_1-t_2)k_1}{1-\ov{t_2}t_1k_1^2}\right|\leq\left|\frac{(t_1-t_2)k_2}{1-\ov{t_2}t_1k_2^2}\right|,\,
\text{ if }\;0<k_1\leq k_2\,\text{ and }\; k_2^2|t_1t_2|<1.
\end{equation}
\end{lemma}
\begin{proof}
Without any loss of generality, we may assume that $t_1t_2\neq0$.
Let $k$ be a real variable and put
\begin{equation*}
F(k)=\left|\frac{(t_1-t_2)k}{1-\ov{t_2}t_1k^2}\right|^2=
\frac{|t_1-t_2|^2k^2}{1+|t_1t_2|^2k^4-2k^2Re(\ov{t_2}t_1)}.
\end{equation*}
It is easy to verify that $F'(k)\geq0$ as $k\in
(0,1/\sqrt{|t_1t_2|})$. Therefore $F(k)$ is an increasing function
on $(0,1/\sqrt{|t_1t_2|})$ and hence (\ref{Eq:dist}) holds.
\end{proof}

\begin{lemma}\label{Th:asynu1}
Let $\mu\in Bel (\de)$ and $p\in \pa \de$. Then, for any given
$\epsilon>0$,
\\
(1) if $\mu\in M(\de)$, then there exists a Beltrami differential
$\nu\in\emu$ such that $\nu$ is an asymptotic extremal  and
$h^*_p(\nu)<h_p(\emu)+\epsilon$;\\
(2) there exists a Beltrami differential $\nu\in\bmu$ such that
$\nu$ is an asymptotic  extremal and $b^*_p(\nu)< b_p(\bmu)+\epsilon$.
\end{lemma}
\begin{proof}We only show the first part for the second part follows from a similar argument.

Case 1. $h_p(\emu)=h(\emu)$.

By Theorem 2 on page 296 of \cite{GL}, there exists a Beltrami differential
$\nu\in\emu$ such that $\nu$ is an asymptotic extremal representative, that is, $h^*(\nu)=h(\emu)$. It obviously yields $h^*_p(\nu)<h_p(\emu)+\epsilon$.

 Case 2.  $h_p(\emu)<h(\emu):=h$.

 By the definition of
boundary dilatation, there exists a Beltrami differential
$\chi(z)\in \emu$ such that
$h_p^*(\chi)<\min\{h_p(\emu)+\epsilon,h\}$.  Let $B(p)=\{z\in \de:\;|z-p|<r\}$ for small $r>0$. Then, when $r$ is sufficiently small,
$|\chi(z)|<\min\{h_p(\emu)+\epsilon,h\}$ in $ B(p)$  almost everywhere.

 Restrict $\chi$ on
$\de\backslash  B(p)$ and regard $[\chi]$ as a point in the \T space  $T(\de\backslash
B(p))$. Then $h([\chi])=h$ (if necessary, let $B(p)$ be smaller). By
Theorem 2 on page 296 of \cite{GL} again, we can choose an asymptotic extremal in
$[\chi]$, say $\nu_1(z)$. Define
\begin{equation*}
\nu(z)=
\begin{cases}
\nu_1(z),& z\in \de\backslash  B(p),\\
\chi(z), &z\in   B(p).\end{cases}
\end{equation*}
Then, $\nu$ is the desired asymptotic extremal in $\emu$.
\end{proof}

\begin{lemma}\label{Th:corfan}
Suppose $\mu$ and $\nu$ be two asymptotically extremal Beltrami
differentials   in their classes in $\atde$ respectively.  If $h([[\mu]])=h([[\nu]])$ and
\begin{equation*}
\sup_{\qdde}\limsup_{n\to\infty}\left |\iint_\de(\mu-\nu)\phi_n \, dxdy\right|>0,
\end{equation*}
then the two geodesics $[[t\mu]]$ and
$[[t\nu]]$ ($0\leq t\leq 1$)  are different.
\end{lemma}
\begin{proof}
It follows readily from Corollary \ref{Th:corvar} by letting $S=\de$.
\end{proof}

\textbf{Proof of Theorem \ref{Th:geodesic}.}
Suppose $\mut$ is not a substantial point in $\atde$. Let $h=h(\mut)$.
 There is  a non-substantial boundary point $q\in
\pa \de$ such that $h_q(\mut)<h$. By Lemma \ref{Th:asynu1},  it is
convenient to assume that $\mu $ is an asymptotic extremal
representative in $\mut$ satisfying  $h_q^*(\mu)<h$.

 By  the definition of boundary dilatation,
we can find a small neighborhood $B(q)$  of $q$ in $\de$ such that
$|\mu(z)|\leq \rho<h$ for some $\rho>0$ in $B(q)$ almost everywhere.
Therefore for any
 $\zeta\in \pa \de\cap  \pa B(q)$, $h_\zeta^*(\mu)\leq \rho$.

Choose  $\delta(z)\in
M(\de)$
 such that  $\|\delta\|_\infty\leq \beta<h-\rho$ and $\delta(z)=0$ when $z\in \de\backslash B(q)$.

Let $\Sigma$ be the collection of the real-valued functions $\sigma(t)$ defined on $[0,h]$ with the following conditions:
\\
(A)  $\sigma$ is  continuous with  $\sigma(0)=0$ and $\sigma(h)=0$,\\
(B) $\frac{|s-t|\rho/h+|\sigma(t)-\sigma(s)|\beta}{1-(s\rho/h+|\sigma(s)|\beta)(t\rho/h+|\sigma(t)|\beta)}
\leq\frac{|s-t|}{1-st},\; s,t\in[0, h]$.

We claim that $\Sigma$ contains uncountably many  elements. At first, let $\sigma$ be a Lipschitz continuous function on
$[0,h]$ with the following conditions,\\
(i) for some small $\alpha>0$, $|\sigma(s)-\sigma(t)|<\alpha |s-t|$, $t,\;s\in [0,h]$,\\
(ii) $\sigma(0)=0$ and $\sigma(h)=0$,\\
(iii) for some small $t_0$ in $(0,h)$,  $\sigma(t)\equiv0$ when $t\geq t_0$.

Secondly, we show that when $t_0$ and $\alpha$ are sufficiently small, $\sigma$ belongs to $\Sigma$, for which it suffices to show that $\sigma$ satisfies the condition (B). Let $t,\;s\in [0,h]$. It is no harm to assume that $t\leq s$.

Case 1. $h\geq t\geq t_0$.

Since $\sigma(s)=\sigma(t)=0$, by Lemma \ref{Th:dist}, we have
\begin{align*}
&\frac{|s-t|\rho/h+|\sigma(t)-\sigma(s)|\beta}{1-(s\rho/h+|\sigma(s)|\beta)(t\rho/h+|\sigma(t)|\beta)}
=\frac{|s-t|\rho/h}{1-st(\rho/h)^2}\\
&\leq\left|\frac{s-t}{1-s t}\right|.
\end{align*}

Case 2. $0\leq t< t_0$.

Put $\gamma=\rho/h+\alpha\beta$ and choose small $\alpha>0$ such that  $\gamma<1$.  On the one hand, since $|\sigma(t)|\leq \alpha t$ and $|\sigma(s)|\leq \alpha s$, it holds that
\begin{align*}
&\frac{|s-t|\rho/h+|\sigma(t)-\sigma(s)|\beta}{1-(s\rho/h+|\sigma(s)|\beta)(t\rho/h+|\sigma(t)|\beta)}
\leq \left|\frac{(s-t)(\rho/h+\alpha\beta)}{1-(s\rho/h+|\sigma(s)|\beta)(t\rho/h+|\sigma(t)|\beta)}\right|\\
&\leq \left|\frac{(s-t)(\rho/h+\alpha\beta)}{1-[\rho/h+\alpha\beta][t_0(\rho/h+\alpha\beta)]}\right|
=\gamma\left|\frac{s-t}{1-t_0\gamma^2}\right|.
\end{align*}
On the other hand, we have
\begin{align*}
\left|\frac{s-t}{1-s  t}\right|\geq\left|\frac{s-t}{1+t_0}\right|.
\end{align*}
When $t_0$ is sufficiently small, we can get
\begin{align*}
\left|\frac{s-t}{1+t_0}\right|\geq \gamma\left|\frac{s-t}{1-t_0\gamma^2}\right|.
\end{align*}
Therefore, when $t_0$ and $\alpha$ are sufficiently small, $\sigma$ satisfies the condition (B).

For a given $\sigma\in \Sigma$, define for $t\in [0,h]$,
 \begin{equation}\label{Eq:nonsub}
 \mu_t(z)=
 \begin{cases}
t\mu(z)/h,&z\in \de\backslash B(q),\\
t\mu(z)/h+\sigma(t)\delta(z),&z\in  B(q).
 \end{cases}
 \end{equation}

\textit{Step 1.} We prove that $\{[[\mu_t]]:t\in [0,h]\}$ is a geodesic connecting $\oo$ and  $\mut$. It is sufficient to verify
that whenever $t,s\in [0,h]$,
\begin{equation}\label{Eq:sst}
d_{AT}([[\mu_t]],[[\mu_s]])=d_H(t,s)=\frac{1}{2}\log\frac{1+|s-t|/(1-st)}{1-|s-t|/(1-st)}.
\end{equation}

Let $f_s:\, \de\to \de$ and $f_t:\, \de\to \de$ be \qc
mappings with Beltrami differentials $\mu_s$ and $\mu_t$
respectively. It is convenient to assume that $t\neq 0$ and $t\neq s$. Set $F_{s,t}=f_s\circ f_t^{-1}$ and assume that the
Beltrami differential of $F_{s,t}$ is $\nu_{s,t}$. Then a simple
computation shows,
\begin{equation*}
\nu_{s,t}\circ
f_t(z)=\frac{1}{\tau}\frac{\mu_s(z)-\mu_t(z)}{1-\ov{\mu_t(z)}\mu_s(z)},
\end{equation*}
where $z=f_t^{-1}(w)$ for $w\in  \de$ and $\tau=\ov{\partial
f_t}/\partial f_t$. We have

\begin{equation}\label{Eq:ft1}
\nu_{s,t}\circ
f_t(z)=\begin{cases}
\frac{1}{\tau}\frac{s-t}{1-  st|\mu(z)|^2/h^2 }\frac{\mu(z)}{h},\,&z\in \de\backslash B(q),\\
\frac{1}{\tau}\frac{(s-t)\mu(z)/h+[\sigma(s)-\sigma(t)]\delta(z)}{1-[s\mu(z)/h+\sigma(s)\delta(z)]
\overline{t\mu(z)/h+\sigma(t)\delta(z)}},\, &z\in B(q).
\end{cases}
\end{equation}
Since $\sigma(t)\in \Sigma$, due to condition (B) we see that
restricted on $f_t(B(q))$,
\begin{equation}\label{Eq:vvst}
\|\nu_{s,t}\|_\infty\leq\frac{|s-t|}{1-st}.
\end{equation}

Suppose  $p\in \pa \de$ is a substantial boundary point for $\mut$. By
Lemma \ref{Th:atslem1} there is a degenerating Hamilton sequence
$\{\psi_n\}\subset Q^1(\de)$ towards $p$ such that
\begin{align*}
h=\lim_{n\to\infty}\iint_\de \mu(z)\psi_n(z)dxdy.
\end{align*}
Then we have
\begin{align*}\label{Eq:qs11}
t=\lim_{n\to\infty}\iint_\de \mu_t(z)\psi_n(z)dxdy.
\end{align*}
On the other hand, it is easy to see that
$h([[\mu_t]])=h^*(\mu_t)=t$ and hence $\mu_t$ is an asymptotic
extremal.
   Therefore, the Beltrami differential $\wt \mu_t$ of  $f^{-1}_t$ is also an asymptotic extremal
where ${\wt\mu}_t=-\mu_t({f}_t^{-1})\ov{\partial
{f}_t^{-1}}/\partial {f}_t^{-1}$. $f_t(p)$ is a
substantial boundary point for $[[\wt\mu_t]]$ and there is a
degenerating Hamilton sequence $\{\wt\psi_n\}\subset Q^1(\de)$
towards $f_t(p)$ such that
\begin{equation*}
\lim_{n\to\infty}\iint_{\de} \wt\mu_t\wt\psi_n(w)dudv=h([[\wt\mu_t]])=t.
\end{equation*}
 Furthermore,
\begin{align}\label{Eq:vst}
&\lim_{n\to\infty}\iint_{\de}
\nu_{s,t}(w)\wt\psi_n(w)dudv=\lim_{n\to\infty}\iint_{\de}\frac{s-t}{1-st}
\frac{{\wt\mu}_t}{t}\wt\psi_n(w)dudv=\frac{s-t}{1-st}.
\end{align}

 In terms of (\ref{Eq:ft1}) and
Lemma \ref{Th:dist}, it is not hard to prove that
$h^*_\zeta(\nu_{s,t})\leq \frac{|s-t|}{1-st}$ when $\zeta\in { \pa
\de\cap\pa (f_t(\de\backslash B(q))}$. Thus, by (\ref{Eq:vvst}),
(\ref{Eq:vst}) and Lemma \ref{Th:atslem1}, it follows  that
$h([[\nu_{s,t}]])=\frac{|s-t|}{1-st}$, $\nu_{s,t}$ is asymptotically
extremal and the equality (\ref{Eq:sst}) holds.

\textit{Step 2.}   We show that, when  $\sigma(t)$ varies over $\Sigma$ and $\delta(z)$ varies
over $ M(\de)$ suitably, respectively, we can get  infinitely many different geodesics.

Firstly, choose $\delta(z)$  in $M(\de)$ such that \begin{equation}\label{Eq:lakic}
\sup_{\qdde}\limsup_{n\to\infty}\left |\iint_\de\delta\vp_n \, dxdy\right|=c>0,
\end{equation}
where the  $supremum$ is over all sequences   $\{\vp_n\}$ in $\qdde$
degenerating towards $q$.

Secondly,  we  choose small $t_0$ in $(0,h)$, small $\alpha>0$   and $\sigma\in \Sigma$ such that $\sigma(t)\equiv0$ when  $t\in [t_0,h]$  and
 $\sigma(t)=\alpha t$  when $t\in [0,t_0/2]$.

\textit{Claim.} When  $\alpha$ varies in a small range,
the   geodesics $[[\mu_t]]$  ($t\in [0,h]$) are mutually different.

Let $\alpha_1$ and $\alpha_2$ be two small different positive numbers and
$\sigma_j(t)=\alpha_j t$  when $t\in [0,t_0]$ ($j=1,2$), respectively.
Now, the corresponding expression of equation (\ref{Eq:nonsub}) is
 \begin{equation*}\mu_t^j(z)=
 \begin{cases}
t\mu(z)/h,&z\in \de\backslash B(q),\\
t\mu(z)/h+\sigma_j(t)\delta(z),&z\in B(q),\; j=1,2.
 \end{cases}
 \end{equation*}

They correspond to
geodesics $G_j=\{[[\mu^j_t]]:t\in [0,h]\}$ ($j=1,2$), respectively.
Note that when $t\in [0,t_0/2]$,
 \begin{equation*}\mu_t^j(z)=
 \begin{cases}
t\mu(z)/h,&z\in \de\backslash B(q),\\
t\mu(z)/h+t\alpha_j\delta(z),&z\in B(q),\; j=1,2.
 \end{cases}
 \end{equation*}
 Define \begin{equation*}\mu^j(z)=
 \begin{cases}
\mu(z)/h,&z\in \de\backslash B(q),\\
\mu(z)/h+\alpha_j\delta(z),&z\in B(q),\; j=1,2.
 \end{cases}
 \end{equation*}

Since
\begin{align*}\sup_{\qdde}&\limsup_{n\to\infty} \left|\iint_\de(\mu^1-\mu^2)\vp_n \, dxdy\right|=
\sup_{\qdde}\limsup_{n\to\infty}\left |\iint_\de(\alpha_1-\alpha_2)\delta\vp_n \, dxdy\right|\\
\geq &|\alpha_1-\alpha_2|c>0,
\end{align*}
by Lemma \ref{Th:corfan}, the
geodesics $G_1$ and $G_2$ are  different.

 If fix small $\alpha>0$ and let $\delta$ vary suitably in $M(\de)$, then we can also get infinitely many geodesics as desired.
The proof of Theorem \ref{Th:geodesic} is completed.

Suppose $\mu$ ($\neq 0$) is a non-Strebel extremal. Then $[t\mu]$ and $[[t\mu]]$ ($t\in [0,1]$) are the geodesics in $\tde$ and $\atde$ respectively. If  $\mu$ is uniquely extremal  in $\emu$  with constant modulus, then the geodesic joining $\emu$ with $[0]$ is unique in $\tde$ \cite{ELi}. Suppose $\mut$ is  a non-substantial point in $\atde$ in addition. Then by Theorem \ref{Th:geodesic} there are infinitely many geodesics joining $\mut$ with $\oo$ in $\atde$.
The example given in Theorem 3.1 of \cite{Fan} provides such a  non-Strebel extremal $\mu$. But Fan's  proof  is lengthy and complicated. Here we give a new example.

\textit{Example.} Let $\phi(z)$ be  holomorphic  on $\ov\de$ except has poles of at most order 2 on $\pa\de$. Assume that $\phi(z)$ has a second-order pole at $z=1$. Then by Reich's result \cite{Re2}, $\mu=k\frac{\ov\phi}{|\phi|}$ ($k\in (0,1)$) is uniquely extremal and $\emu$ is a non-Strebel point in $\tde$.   In addition, it is easy to check that  $h^*_\zeta(\mu)=0$ if $\zeta\in \pa\de$ is neither a pole  nor a zero  of $\phi(z)$. Therefore, $\mut$ is not a substantial point in $\atde$.

\begin{centering}\section{\!\!\!\!\!{. }Straight lines containing two points}\label{S:line}\end{centering}

We say that $\mu$ is a non-Strebel extremal if  it is an extremal representative in the non-Strebel point $\emu$ (or $\emu_Z$).
The following lemma says that a non-Strebel extremal  as an asymptotic extremal representative always exists in a class $\mut$.

\begin{lemma}\label{Th:nonstrebel}

Let $\mu\in Bel (S)$. Then,
\\
(1) if $\mu\in M(S)$, then there exists a Beltrami differential
$\nu\in\mut$ such that $\nu$ is a non-Strebel extremal;\\
(2) there exists a Beltrami differential $\nu\in\azmu$ such that
$\nu$ is a non-Strebel extremal.
\end{lemma}
\begin{proof}
We only show the first part (1).

By Theorem 2  on page 296 of \cite{GL}, there is an asymptotic extremal representative in $\mut$, say $\mu$, such that $h(\mut)=h^*(\mu)$.
If $h^*(\mu)=0$, let $\nu$ be identically zero. If $h^*(\mu)>0$, put
\begin{equation*}\nu(z)=
\begin{cases}
\mu(z),& |\mu(z)|\leq h^*(\mu),\\
h^*(\mu)\mu(z)/|\mu(z)|,& |\mu(z)|> h^*(\mu).
\end{cases}
\end{equation*}
In either case, it is easy to verify that $\nu\in \mut$  and is a non-Strebel extremal.
\end{proof}
Let $\mu$ ($\neq0$) be a non-Strebel extremal. Then  there are infinitely many straight
lines containing $[0]$ and $\emu$ in $\tde$. However, it
cannot be directly inferred that there are infinitely many straight
lines containing $\oo$ and $\mut$ in $\atde$ since the topologies
induced by  metrics in $\tde$ and $\atde$ are essentially different.

\textbf{Proof of Theorem \ref{Th:line}.} Up to an isometry of $\atde$,
it suffices to prove that for any $\mut$ ($\neq \oo)$ in $\atde$, there
are infinitely many straight lines passing through $\mut$ and $\oo$. By Lemma \ref{Th:nonstrebel}, we choose a non-Strebel extremal representative in $\mut$, say $\mu$.
 Then $k_0(\emu)=h(\mut)=h^*(\mu):=h$.

\textit{Case 1.} $\mut$ is a substantial point.

Fix a boundary point
$p\in \pa \de$. Let  $B(p)=\{z\in \de:\;|z-p|<r\}$ for small $r>0$ and  $E=\de\backslash B(p)$.  Define for $t\in(-1,1)$,
\begin{equation}
\mu_t(z):=\begin{cases}
t\mu(z)/h,&z\in \de,\;|t|\leq h,\\
t\mu(z)/h,&z\in \de\backslash E,\; |t|>h,\\
sgn(t)\mu(z),& z\in E,\; |t|>h.
\end{cases}
\end{equation}

We prove that $G_E=\{[[\mu_t]]:\;t\in (-1,1)\}$ is   a straight
line passing through $\oo$ and $\mut$. Note that  $G_E$  differs from
 the straight line $G[\mu]=\{[[t\mu/h]]:\;t\in (-1,1)\}$ only when $|t|>h$. It is sufficient to show the
following two points: for any given $\rho\in (h,1)$,
\\
(i) $\dat([[\mu_{-\rho}]],[[\mu_{\rho}]])=d_H(-\rho,\rho)$;
\\ (ii) $\{[[\mu_t]]:\;t\in [0,\rho]\}$ and  $\{[[\mu_t]]:\;t\in [-\rho,0])\}$ are two geodesics
and \begin{equation}\label{Eq:rho1}
\dat([[\mu_{-\rho}]],\oo)=\dat(\oo,[[\mu_\rho]])=\frac{1}{2}d_H(-\rho,\rho).\end{equation}

(i) is relatively clear since on $E$, $|
\rho\mu(z)/|\rho||<|\rho\mu(z)/h|$ for $\rho>h$, so is
(\ref{Eq:rho1}). Due to symmetry, for (ii), it suffices to show
that $\{[[\mu_t]]:\;t\in [0,\rho]\}$ is a geodesic. This is reduced
to prove that $\{[[\mu_t]]:\;t\in [0,h]\}$ and  $\{[[\mu_t]]:\;t\in
[h,\rho])\}$ are two geodesics, and \begin{equation*}\label{Eq:rho2}
\dat(\oo,[[\mu_h]])+\dat([[\mu_{h}]],[[\mu_{\rho}]])=d_H(0,\rho).\end{equation*}
In such a case, we only need to check that $\{[[\mu_t]]:\;t\in
[h,\rho])\}$ is a geodesic with length $d_H(h,\rho)$. In fact, when
$\rho\leq t<s\leq h$,
using  the previous notation, we have
\begin{equation}\label{Eq:atsft1}
\nu_{s,t}\circ
f_t(z)=\begin{cases}
\frac{1}{\tau}\frac{s-t}{1-  st|\mu(z)|^2/h^2 }\frac{\mu(z)}{h},\,&z\in \de\backslash E,\\
0,\,
&z\in E.
\end{cases}
\end{equation}
Now, it is evident that
\[d_{AT}([[\mu_t]],[[\mu_s]])=\frac{1}{2}\log\frac{1+h([[\nu_{s,t}]])}{1-h([[\nu_{s,t}]])}=d_H(t,s)\; t,s\in [h,\rho].\]

Comparing $[[\mu_t]]$ with $[[t\mu/h]]$  as $|t|>h$, we find that
$[[\mu_t]]$ is no longer a substantial point since the boundary
 points in the interior of  $\pa \de\cap \pa E$ are
no longer substantial ones. Therefore, when the boundary point $p$ or the neighborhood
$B(p)$ varies, we get infinitely many different straight lines.

\textit{Case 2.} $\mut$ is not a substantial point.

By Theorem
\ref{Th:geodesic}, there are infinitely many geodesics connecting
$\oo$ and $\mut$. We can then  extend these geodesics to straight
lines by uniformly defining,
$[[\mu_t]]=[[t\mu/h]]$ for $t\in(-1,0)\cup(h,1)$. The verification  is similar to Case 1 and is omitted here.
The completes the proof of Theorem \ref{Th:line}.

\begin{centering}\section{\!\!\!\!\!{. }  Relationship
on  substantial boundary points for points along a geodesic}\label{S:property}\end{centering}

In this section, we investigate the relationship
on  substantial boundary points for the points along a
geodesic. We have the following result.
\begin{theorem}\label{Th:subprop}
 Suppose $h(\mut)=h\in (0,1)$ and $\{[[\mu_t]]:\,t\in (0,h)\}$ is a
geodesic connecting $\oo$ and $\mut$ such that
$\dat(\oo,[[\mu_t]])=d_H(0, t)$ for $t\in (0,h)$. If $p\in \pa \de$ is
a substantial boundary point for $\mut$, then $p$ is a substantial
boundary point for all $[[\mu_t]]$, $t\in (0,h)$.
\end{theorem}
\begin{proof}

Let $h=h(\mut)$.
 Given $t\in (0,h)$, let $\nu(z)\in [[\mu_t]]$ be an asymptotic
 extremal representative  with $h^*(\nu)=t$. We need to show that $h^*_p(\nu)=t$.

Let $f:\;\de\to \de$ and $g:\;\de\to \de$ be the  \qc mappings with
Beltrami differentials $\mu$ and $\nu$, respectively. Let $\Lambda$
  be an asymptotic extremal \qc mapping
in the asymptotically equivalence class $[[f\circ
g^{-1}]]$. Assume that $\lambda(w)$ is the  Beltrami differential of
$\Lambda$ where $w=g(z)$, and  $h([[\lambda]])=h^*(\lambda)=\alpha$.
 Put $F=\Lambda\circ g$. Then $F$ and $f$ is
 asymptotically equivalent and
\begin{equation}\label{Eq:muf}
\mu_F=\frac{\nu+\lambda\circ g \cdot
\tau}{1+\overline{\nu}\cdot\lambda\circ g \cdot \tau},
\end{equation}
where $\tau=\overline{\pa_z g}/\pa_z g$. Since
\begin{equation*}
\dat(\oo,\mut)=\dat(\oo, [[\mu_t]])+\dat([[\mu_t]],
\mut),\end{equation*} we have
\begin{equation*}
\frac{1}{2}\log\frac{1+h}{1-h}=\frac{1}{2}\log\frac{1+t}{1-t}+\frac{1}{2}\log\frac{1+\alpha}{1-\alpha},
\end{equation*}
equivalently,
\begin{equation*}
\frac{1+h}{1-h}=\frac{1+t}{1-t}\cdot\frac{1+\alpha}{1-\alpha}.
\end{equation*}
This leads to
\begin{equation}\label{Eq:talpha}
h=\frac{t+\alpha}{1+t\alpha}.
\end{equation}
On the other hand, by (\ref{Eq:muf}) we have
\begin{align*}
|\mu_F|=\biggl|\frac{\nu+\lambda\circ g \cdot
\tau}{1+\overline{\nu}\cdot\lambda\circ g \cdot \tau}\biggr| \leq
\frac{|\nu|+|\lambda\circ g|}{1+|\nu||\lambda\circ g |}.
\end{align*}
Therefore, by the definition of boundary dilatation, we get
\begin{align}\label{Eq:ak}
h=h_p(\mut)\leq h^*_p(\mu_F)\leq \frac{h^*_p(\nu)
+h^*_{g(p)}(\lambda)}{1+h^*_p(\nu)h^*_{g(p)}(\lambda)}\leq
\frac{t+\alpha}{1+t\alpha}.
\end{align}
Notice that $h^*_p(\nu)\leq t$ and $h^*_{g(p)}(\lambda)\leq \alpha$.
Combining (\ref{Eq:talpha}) and (\ref{Eq:ak}), we   must have
$h^*_p(\nu)=t$. The concludes the proof.\end{proof}

The following corollary follows  immediately.
\begin{cor}
If $\mut$ is a substantial point in $\atde$, then every point in a
geodesic connecting $\oo$ and $\mut$
 is a substantial point.
\end{cor}
There is a natural projection from $\tde$ onto $\atde$,
 \begin{align*}\pi:\tde &\to \atde,\\
 \emu&\to \mut.
 \end{align*}
If $\mu$ is a non-Strebel extremal, then the projection of any
geodesic connecting $[0]$ and $\emu$ in $\tde$ under  $\pi$ is a
geodesic connecting $\oo$ and $\mut$ in $\atde$. Therefore, we have
the following corollary.
\begin{cor}
If $\mu$ is a non-Strebel extremal and $p\in \pa \de$ is a substantial
boundary point for $\emu$,  then  $p$ is a substantial boundary
point for all points in a geodesic connecting $[0]$ and $\emu$ in
$\tde$.
\end{cor}

\begin{rem}
If $p\in \de$ is not a substantial boundary point for $\mut$, it is
possible that $p$ is a substantial boundary point for some point
(hence for infinitely many points) in the geodesic connecting $\oo$ and
$\mut$  (see Case 1 in  the proof of Theorem \ref{Th:line} in Section
\ref{S:line}).
\end{rem}

\begin{centering}\section{\!\!\!\!\!{. }An example for geodesics joining a substantial point with the basepoint}\label{S:example}\end{centering}
The situation on the geodesics joining  a  substantial point with the
basepoint is not clear.
The difficulty can be seen from Theorem \ref{Th:subprop}, for which the method used in the proof of Theorem \ref{Th:geodesic}  does not apply for  a substantial point. However, one must
not expect that the geodesic passing through a substantial point and
the basepoint is necessarily unique.  We now construct certain
counterexample to show how it is.

 We divide the construction
of example into three steps.

\textit{Step 1.} At first, we introduce an example, which we
describe below,
 every point $\zeta\in \pa \de$ is a
substantial boundary point.  The example was shown  in \cite{FS} and
was said to be due to Reich by an oral communication. For the sake
of clarity and completeness, here we include the detail for
construction which was demonstrated in \cite{FS}.

 \textit{Example.} Let $\phi_n$ be the
sequence defined by
\begin{equation*}
\phi_n(z):=\frac{(n+2)z^n}{2\pi}.
\end{equation*}
For a fixed number $k$, $0<k<1$, we define, for every $n\in\mathbb{N}$,
\begin{equation*}
\kappa_n(z):=k\frac{\bar z^n}{|z|^n}.
\end{equation*}
Then for $0\leq \rho_1\leq \rho_2\leq 1$
\begin{equation*}
\iint_{\rho_1<|z|<\rho_2}\kappa_n\phi_ndxdy=\frac{(n+2)k}{2\pi}\int_0^{2\pi}\int_{\rho1}^{\rho_2}
r^{n+1}drd\theta=k(\rho_2^{n+2}-\rho_1^{n+2})
\end{equation*}
and
\begin{equation*}
\iint_{\rho_1<|z|<\rho_2}|\phi_n|dxdy=\frac{n+2}{2\pi}\int_0^{2\pi}\int_{\rho1}^{\rho_2}
r^{n+1}drd\theta=\rho_2^{n+2}-\rho_1^{n+2}.
\end{equation*}
Choose a number $n_1\in \mathbb{N}$. Then there is a number $r_1$, $0<r_1<1$, with
\begin{equation*}
r_1^{n_1+2}>1-\frac{1}{2}
\end{equation*}
and we compute
\begin{equation*}
\iint_{0<|z|<r_1}\kappa_{n_1}\phi_{n_1}dxdy=kr_1^{n_1+2}>k(1-\frac{1}{2})
\end{equation*}
and
\begin{equation*}
\iint_{r_1<|z|<1}|\phi_{n_1}|dxdy=1-r_1^{n_1+2}<\frac{1}{2}.
\end{equation*}
Next we choose $n_2>n_1$ such that
\begin{equation*}
r_1^{n_2+2}<\frac{1}{2^2}.
\end{equation*}
Then there is a number $r_2$, $r_1<r_2<1$, such that
\begin{equation*}
r_2^{n_2+2}>1-\frac{1}{2^2}
\end{equation*}
and we may also have $r_2>r_1+\frac{1-r_1}{2}$. We compute
\begin{equation*}
\iint_{|z|<r_1}|\phi_{n_2}|dxdy=r_1^{n_2+2}<\frac{1}{2^2},
\end{equation*}
\begin{equation*}
\iint_{r_1<|z|<r_2}\kappa_{n_2}\phi_{n_2}dxdy=k(r_2^{n_2+2}-r_1^{n_2+2})
>k(1-\frac{1}{2^2}-\frac{1}{2^2})=k(1-\frac{1}{2})
\end{equation*}
and
\begin{equation*}
\iint_{r_2<|z|<1}|\phi_{n_2}|dxdy=1-r_2^{n_2+2}<\frac{1}{2^2},
\end{equation*}
Proceeding this construction, we get a sequence $n_j$ ($n_j\to
\infty$) and a sequence $r_j, $ $r_1<r_2<\cdots<1,$ $r_j\to 1$
($j\to \infty$). Furthermore, because of $r_{j-1}^{n_j+2}<1/2^j$,
$r_j^{n_j+2}>1-1/2^j$, we have for $j\geq2$
\begin{equation}\label{Eq:fs1}
\iint_{|z|<r_{j-1}}|\phi_{n_j}|dxdy=r_{j-1}^{n_j+2}<\frac{1}{2^j},
\end{equation}
\begin{equation}\label{Eq:fs2}
\iint_{r_{j-1}<|z|<r_j}\kappa_{n_j}\phi_{n_j}dxdy=k(r_j^{n_j+2}-r_{j-1}^{n_j+2})
>k(1-\frac{1}{2^j}-\frac{1}{2^j})=k(1-\frac{1}{2^{j-1}}),
\end{equation}
\begin{equation}\label{Eq:fs3}
\iint_{r_j<|z|<1}|\phi_{n_j}|dxdy=1-r_j^{n_j+2}<\frac{1}{2^j}.
\end{equation}
Clearly, $\{\phi_n\}\subset \qde$ is a degenerating sequence in
$\de$. Set $E_j=\{z:r_{j-1}\leq|z|<r_j\}$ for $j\geq1$ where let
$r_0=0$. Define
\begin{equation}\label{Eq:kappa1}
\kappa(z):=\begin{cases}
\kappa_{n_1}(z),&z\in E_1,\\
\kappa_{n_2}(z),&z\in E_2,\\
\;\,\;\vdots\\
\kappa_{n_j}(z),&z\in E_j,\\
\;\,\;\vdots
\end{cases}
\end{equation}
Then  $\kappa(z)$ has constant modulus $k$. Regard $\kappa(z)$ as
the complex dilatation of a \qc self-mapping $f$ of $\de$.
 By (\ref{Eq:fs1}),
(\ref{Eq:fs2}) and (\ref{Eq:fs3}) we have
\begin{equation}\label{Eq:sakan1}
\begin{split}
Re\iint_\de\kappa\phi_{n_j}dxdy&\geq Re\iint_{r_{j-1}<|z|<r_j}\kappa_{n_j}\phi_{n_j}dxdy\\
&- \iint_{|z|\leq r_{j-1}\text{ or}\;r_j\leq |z|<1}k|\phi_{n_j}|dxdy
\geq k(1-\frac{1}{2^{j-1}})-\frac{k}{2^{j-1}}.
\end{split}
\end{equation}
Thus, we have
\begin{equation*}
\lim_{j\to\infty}Re\iint_\de \kappa\phi_{n_j}dxdy=k
\end{equation*}
and hence $\phi_{n_j}$ is a Hamilton sequence for the extremal complex dilatation $\kappa$. Moreover,
as Fehlmann and Sakan
noted in their paper, by Theorem 1.1 in \cite{FS}, every $\zeta\in \pa\de$ is a substantial
 boundary point for $[\kappa]$.

\textit{Step 2.} With some  modification on  $\kappa$, we define a new  complex dilatation as follows,
\begin{equation*}
\mu(z):=\begin{cases}
\alpha\kappa(z),&z\in E_{2m-1},\\
\beta\kappa(z),&z\in E_{2m},
\end{cases}
\end{equation*}
where $m\geq1$ and the constants $\alpha,\beta\in [0,1/k)$.

\textbf{Claim.} $\mu$ is extremal  and $k_0([\mu])=\max\{\alpha k,\beta k\}$. Moreover, every
$\zeta\in \pa\de$ is a substantial boundary point for $[\mu]$.
\begin{proof} If $\alpha=\beta$, then $\mu=\alpha\kappa$ and the claim is a fortiori.

 Let $\alpha <\beta$ first.  By the reasoning deriving (\ref{Eq:sakan1}), we have
\begin{equation*}
\begin{split}
Re\iint_\de\mu\phi_{n_{2m}}dxdy&\geq \beta Re\iint_{E_{2m}}\kappa_{2m}\phi_{n_{2m}}dxdy\\
&- \beta\iint_{\de\backslash E_{2m}}  k|\phi_{n_{2m}}|dxdy \geq
\beta k(1-\frac{1}{2^{2m-1}})- \frac{\beta k}{2^{2m-1}}.
\end{split}
\end{equation*}
Thus, we get
\begin{equation*}
\lim_{m\to\infty}Re\iint_\de \mu\phi_{n_{2m}}dxdy=\beta k.
\end{equation*}
Hence $\mu$ is extremal with $\|\mu\|_\infty=\beta k$ and
$\phi_{n_{2m}}$ is a degenerating Hamilton sequence. Similarly, if
$\alpha >\beta$, then $\mu$ is extremal with $\|\mu\|_\infty=\alpha
k$ and  $\phi_{n_{2m-1}}$ is a degenerating Hamilton sequence.
Anyway, the aforementioned reason implies that
  every
$\zeta\in \pa\de$ is a substantial boundary point for $[\mu]$ or $\mut$. This Claim is proved.
\end{proof}

\textit{Step 3.} Fix $\alpha\in (0,1)$ and $\beta=1$. Then $\mu$ is
extremal with $\|\mu\|_\infty=k$ and
 $\mut$ is a substantial point in $\atde$. We construct infinitely many
geodesics connecting  $\mut$  and the basepoint.

Let $\Sigma'$  be the collection of the real-valued functions $\sigma(t)$ defined on $[0,k]$  with the following conditions:
\\
(A)  $\sigma$ is  continuous with $\sigma(0)=0$ and $\sigma(k)=k$,\\
(B) $\frac{|\sigma(s)-\sigma(t)|\alpha
}{|1-\sigma(t)\sigma(s)\alpha^2|} \leq\frac{|s-t|}{1-st}$, $t,s\in [0,k]$.\\
Since $0<\alpha<1$, it is easy to verify that  $\Sigma'$ contains uncountably
many  elements.

Given $\sigma \in \Sigma'$, define for $t\in[0,k]$,
\begin{equation}\label{Eq:new1}
\mu_t(z):=\begin{cases}
\sigma(t)\mu(z)/k,&z\in E_{2m-1},\,m\geq 1,\\
t\mu(z)/k,&z\in E_{2m},\,m\geq 1.
\end{cases}
\end{equation}
One easily proves that   $\{[[\mu_t]]:t\in
[0,k]\}$ is a geodesic connecting $\oo$ and  $\mut$.

Fix some $t_0$ in $(0,k)$. Choose
 $\sigma(t)\in \Sigma'$ such that  $\sigma(t)=\lambda t$  when $t\in [0,t_0]$ where  $\lambda\in (0,1)$ is sufficiently small.
We show that for different $\lambda$, these geodesics are mutually
different.

Let $\lambda_1$,  $\lambda_2\in (0,1)$  ( $\lambda_1>\lambda_2$) be small and
$\sigma_j(t)=\lambda_j t$  when $t\in [0,t_0]$ ($j=1,2$), respectively.
Now, on $[0,t_0]$ the corresponding expression of equation (\ref{Eq:new1}) is
\begin{equation}\label{Eq:new2}
\mu^j_t(z):=\begin{cases}
\lambda_jt\mu(z)/k,&z\in E_{2m-1},\,m\geq 1,\\
t\mu(z)/k,&z\in E_{2m},\,m\geq 1.
\end{cases}
\end{equation}

They correspond to
geodesic segments $G_j=\{[[\mu^j_t]]:t\in [0,t_0]\}$ ($j=1,2$), respectively.

 Define
 \begin{equation*}\label{Eq:new3}
\mu^j(z):=\begin{cases}
\lambda_j\mu(z)/k,&z\in E_{2m-1},\,m\geq 1.\\
\mu(z)/k,&z\in E_{2m},\,m\geq 1.
\end{cases}
\end{equation*}
Then
\begin{equation*}
\mu^1-\mu^2=\begin{cases}
(\lambda_1-\lambda_2)\mu(z)/k,&z\in E_{2m-1},\,m\geq 1,\\
0,&z\in E_{2m},\,m\geq 1.
\end{cases}
\end{equation*}

Since
\begin{align*}&\lim_{m\to\infty} \iint_\de(\mu^1-\mu^2)\phi_{n_{2m-1}} \, dxdy
=\frac{1}{k}(\lambda_1-\lambda_2)\lim_{m\to\infty}\iint_\de \mu\phi_{n_{2m-1}}dxdy\\
&=\lambda_1-\lambda_2>0,
\end{align*}
by Lemma \ref{Th:corfan}, the
geodesic segments $G_1$ and $G_2$ are  different.

The example serves  to give  infinitely many geodesics connecting
 the infinitesimal substantial point $\azmu$ and the basepoint in
$\azde$ as well.

In an infinitely dimensional \T space, there always exist closed
geo\-de\-sics and the spheres are not convex due to Li's work
\cite{Li3} (also see \cite{ELi}). Here a closed geodesic means to be locally shortest. As a byproduct of the
example, the following result in the asymptotic \T space is
fairly direct.
\begin{theorem}\label{Th:close}
There exist closed geodesics in the universal asymptotic \T space $\atde$ and hence
the spheres in $\atde$ are not convex.
\end{theorem}
\begin{proof}
Define
\begin{equation*}
\eta_1(z):=\begin{cases}
\kappa(z),&z\in E_{2m-1}, \,m\geq 1,\\
0,&z\in E_{2m},\,m\geq 1,
\end{cases}
\end{equation*}

\begin{equation*}
\eta_2(z):=\begin{cases}
0,&z\in E_{2m-1},\,m\geq 1,\\
\kappa(z),&z\in E_{2m},\,m\geq 1,
\end{cases}
\end{equation*}
$\eta_3(z)=-\eta_1(z)$ and $\eta_4(z)=-\eta_2(z)$. Let
$R=\frac{1}{2}\log \frac{1+k}{1-k}$. It is easy to derive that
\begin{equation*}\dat\oo,[[\mu_j]]=R, \;j=1,2,3,4,\;
2R=\dat([[\eta_1]],[[\eta_3]])=\dat([[\eta_2]],[[\eta_4]])
\end{equation*}
and
\begin{equation*}
R=\dat([[\eta_1]],[[\eta_2]])=\dat([[\eta_2]],[[\eta_3]])=\dat([[\eta_3]],[[\eta_4]])=\dat([[\eta_4]],[[\eta_1]]).
\end{equation*}

Define for $t\in[0,k]$
\begin{equation*}
\mu_t(z):=\begin{cases}
\sigma(t)\kappa(z)/k,&z\in E_{2m-1},\,m\geq 1,\\
t\kappa(z)/k,&z\in E_{2m},\,m\geq 1,
\end{cases}
\end{equation*}
where $\sigma(t)=\frac{k-t}{1-tk}$ as $t\in [0,k]$. Using  the same
notation as in the proof of  Theorem \ref{Th:geodesic}, we have
\begin{equation}\label{Eq:atsft1}
\nu_{s,t}\circ f_t(z)=\begin{cases}\frac{1}{\tau}
\frac{\sigma(s)-\sigma(t)}{1-
\sigma(s)\sigma(t)}\frac{\kappa(z)}{k},\,&z\in E_{2m-1},\,m\geq 1,\\
\frac{1}{\tau}\frac{s-t}{1-  st }\frac{\kappa(z)}{k},\, &z\in
E_{2m},\,m\geq 1.
\end{cases}
\end{equation}
Observe that
\begin{equation*}
\biggl|\frac{\sigma(s)-\sigma(t)}{1-
\sigma(s)\sigma(t)}\frac{\kappa(z)}{k}\biggr
|=\biggl|\frac{\sigma(s)-\sigma(t)}{1- \sigma(s)\sigma(t) }\biggr|=
\biggl|\frac{s-t}{1- st }\biggr|=\biggl|\frac{s-t}{1-  st
}\frac{\mu(z)}{k}\biggl|,\,t,\,s\in [0,k].
\end{equation*}
It is not hard to prove that that whenever $t,\,s\in [0,k]$,
\begin{equation*}
d_T([\mu_t],[\mu_s])=\dat([[\mu_t]],[[\mu_s]])=d_H(t,s).\end{equation*}
Hence, $\{[\mu_t]:\;t\in[0,k]\}$ is a geodesic connecting $[\eta_1]$
and $[\eta_2]$ in the universal \T space $\tde$ as well as
 $\{[[\mu_t]]:\;t\in[0,k]\}$ is a geodesic connecting $[[\eta_1]]$ and $[[\eta_2]]$ in $\atde$. Similarly, one can construct
the  geodesic connecting $\eta_2$ and $\eta_3$, and so on. Thus, we construct closed geodesics in $\tde$ and $\atde$
  simultaneously. In particular, the latter is the image of the former under the natural projection
  $\pi$. Moreover, all points in the closed  geodesic are substantial ones.

Consider the sphere centered at $[[\eta_1]]$ and with radius $R$ in $\atde$. By the construction, there are two geodesics connecting $[[\eta_2]]$ and $[[\eta_4]]$. One is $[[\eta_2]]\to [[\eta_1]]\to [[\eta_4]]$ which is located inside the sphere; the other is $[[\eta_2]]\to [[\eta_3]]\to [[\eta_4]]$ which is located outside the sphere.
Now  it is clear that the sphere is not convex.

\end{proof}

One can check that the geodesic joining  $[\kappa]$ (defined by (\ref{Eq:kappa1})) with $[0]$ in $\tde$ is not unique. However, it is not clear up to present
 whether the geodesic connecting the
 substantial point
 $[[\kappa]]$ and $\oo$ in $\atde$ is unique.

\begin{centering}\section{\!\!\!\!\!{. }Geodesics and straight lines in the tangent space}\label{S:tangent}\end{centering}
The following theorem is the counterpart of Theorem \ref{Th:geodesic} in $\azde.$
\begin{theorem}
Suppose $\azmu$ is not a substantial point in $\azde$, i.e.,
$\azmu\in\azde\backslash ISP$. Then there are infinitely many
 geodesics connecting $\azmu$
and the
basepoint $\oaz$.
\end{theorem}

\begin{proof}
 Let $b=b(\azmu)$.
Since  $\azmu$ is not substantial point,
there is a point $q\in
\pa \de $ which is not a substantial boundary point for $\azmu$.
By Lemma \ref{Th:asynu1}, we may assume that $\mu $ is
an asymptotic extremal representative in $\azmu$ such that
$b^*_q(\mu)<b$.

 By  the definition of boundary dilatation,
we can find a small neighborhood $B(q)$  of $q$ in $\de$ such that
$|\mu(z)|\leq \rho<b$ for some $\rho>0$ in $B(q)$ almost everywhere.
Therefore for any
 $\zeta\in \pa \de\cap \pa B(q)$, $b_\zeta^*(\mu)\leq \rho$.

 Choose  $\delta(z)\in
Bel(\de)$
 such that  $\|\delta\|_\infty\leq \beta<b-\rho$ and $\delta(z)=0$ when $z\in \de\backslash B(q)$.

Let $\Sigma''$ be the collection of the real-valued functions $\sigma(t)$ defined on $[0,b]$ with the following conditions:
\\
(A)  $\sigma$ is  continuous with  $\sigma(0)=0$ and $\sigma(b)=0$,\\
(B) $|s-t|\rho/b+|\sigma(t)-\sigma(s)|\beta
\leq|s-t|,$ $t,\;s\in [0,b]$.

Since $\rho<b$ and $\beta <b-\rho$,  $\Sigma''$ contains uncountably many  elements. In fact, if $\sigma$ is a Lipschitz continuous function on $[0,b]$ with the following conditions,\\
(i) for some small $\alpha>0$, $|\sigma(s)-\sigma(t)|<\alpha |s-t|$, $t,\;s\in [0,b]$,\\
(ii) $\sigma(0)=0$ and $\sigma(b)=0$,\\
(iii)  $\rho/b+\alpha \beta<1$,\\
then $\sigma\in \Sigma''$.

Given $\sigma\in \Sigma''$,
define for $t\in [0,b]$,
 \begin{equation}\label{Eq:infnonsub}
 \mu_t(z)=
 \begin{cases}
t\mu(z)/b,&z\in \de\backslash B(q),\\
t\mu(z)/b+\sigma(t)\delta(z),&z\in  B(q).
 \end{cases}
 \end{equation}
  We show that  $[[\mu_t]]_{AZ}$ ($t\in [0,b]$) is a geodesic.
 It is sufficient to verify
that
\begin{equation}\label{Eq:infsst}
\|[[\mu_s-\mu_t]]_{AZ}\|=|s-t|,\; t,\,s\in [0,b].
\end{equation}
At first, it is obvious that
\begin{align*}
\|\mu_s-\mu_t\|_\infty=|s-t|.
\end{align*}

Suppose $p\in \pa \de$ is a substantial boundary point for $\azmu$. By
Lemma \ref{Th:atslem1} there is a degenerating Hamilton sequence
$\{\psi_n\}\subset Q^1(\de)$ towards $p$ such that
\begin{align*}
b=\lim_{n\to \infty}\iint_\de \mu(z)\psi_n(z)dxdy.
\end{align*}
Therefore, we have
\begin{align*}
s-t=\lim_{n\to \infty}\iint_\de [\mu_s(z)-\mu_t(z)]\psi_n(z)dxdy,
\end{align*}
which implies the equality (\ref{Eq:infsst}).

It remains to show that there are infinitely many geodesics passing through $\azmu$ and $\oaz$
when  $\sigma$ varies over $\Sigma''$ and $\delta(z)$ varies
over $ Bel(\de)$ suitably, respectively.

Choose $\delta(z)$  in $Bel(\de)$ such that (\ref{Eq:lakic}) holds.
Fix a  small $t_0$ in $(0,b)$.  Choose $\sigma\in \Sigma''$ such that $\sigma(t)\equiv0$ when $t\geq t_0$  and
 $\sigma(t)=\alpha t$  when $t\in [0,t_0/2]$ where $\alpha>0$ satisfying  $\rho/b+\alpha \beta<1$.
Note that when $t\in [0,t_0/2]$,
 \begin{equation*}\mu_t(z)=
 \begin{cases}
t\mu(z)/b,&z\in \de\backslash B(q),\\
t\mu(z)/b+t\alpha\delta(z),&z\in B(q).
 \end{cases}
 \end{equation*}

Due to the equality (\ref{Eq:lakic}),
the   geodesics  $G_\alpha=\{[[\mu_t]]_{AZ}:t\in [0,b]\}$  are mutually different when $\alpha$ varies in a small range.

If fix small $\alpha>0$ and let $\delta$ vary suitably in $Bel(\de)$, then we can also get infinitely many geodesics as required.

\end{proof}

The  counterpart of Theorem \ref{Th:line} in the infinitesimal setting follows from an almost same argument.
\begin{theorem}\label{Th:infline}
For any two points in $\azde$,  there
     are   infinitely many straight lines containing them.
\end{theorem}

The following is the  infinitesimal version of Theorem
\ref{Th:subprop}.
\begin{theorem}\label{Th:infsubprop}
 Suppose $b(\azmu)=b\in (0,+\infty)$ and $\{[[\mu_t]]_{AZ}:\,t\in (0,b)\}$ is a
geodesic connecting $\oaz$ and $\azmu$ such that
$d_{AZ}(\oaz,[[\mu_t]]_{AZ})=t$ for $t\in (0,b)$. If $p\in \pa \de$ is a
substantial boundary point for $\azmu$, then $p$ is a substantial
boundary point for all $[[\mu_t]]_{AZ}$, $t\in (0,b)$.
\end{theorem}

At last, we end the paper with the infinitesimal version of Theorem \ref{Th:close}.
\begin{theorem}
There exist closed geodesics in the tangent space $\azde$ and hence
the spheres in $\azde$ are not convex.
\end{theorem}

\renewcommand\refname{\centerline{\Large{R}\normalsize{EFERENCES}}}
\medskip
\addcontentsline{toc}{section}{\bf\large{R}\normalsize{EFERENCES}}

\noindent \textit{Guowu Yao}\\
Department of Mathematical Sciences\\
  Tsinghua University\\Beijing,  100084,  People's Republic of
  China \\
  E-mail: \texttt{gwyao@math.tsinghua.edu.cn}
\end{document}